\documentclass[a4paper,11pt]{article}

\usepackage{lscape}
\usepackage{mathtools}
\usepackage{comment}
\usepackage{appendix}
\usepackage{float}
\usepackage{enumitem}
\usepackage{graphics}
\usepackage{graphicx}
\usepackage{epsfig}
\usepackage[T1]{fontenc}
\usepackage{amsmath}
\usepackage{amsthm}
\usepackage{algorithm}
\usepackage{algorithmic}
\usepackage{subfig}
\usepackage{soul}
\usepackage{url} 
\usepackage{makeidx} 
\usepackage{amsfonts} 
\usepackage{listings} 
\usepackage{verbatim}
\usepackage{cite}
\usepackage{hyperref} 
\usepackage{array} 
\hypersetup{
    colorlinks=true,       
    allcolors = blue
}

\usepackage{color}

\newtheorem{theorem}{Theorem}

\newtheorem{lemma}{Lemma}

\newtheorem{proposition}{Proposition}

\makeatletter

\newcommand{\Rmnum}[1]{\expandafter\@slowromancap\romannumeral #1@}
\makeatother

\begin{document}

\title{Scheduling Post-Disaster Repairs in Electricity Distribution Networks}
\author{Yushi Tan, Feng Qiu, Arindam K. Das, Daniel S. Kirschen, \\ Payman Arabshahi, Jianhui Wang}
\date{}

\maketitle

\section*{Abstract}

Natural disasters, such as hurricanes, earthquakes and large wind or ice storms, typically require the repair of a large number of components in electricity distribution networks. Since power cannot be restored before these repairs have been completed, optimally scheduling the available crews to minimize the cumulative duration of the customer interruptions reduces the harm done to the affected community. Considering the radial network structure of the distribution system, this repair and restoration process can be modeled as a scheduling problem with soft precedence constraints. As a benchmark, we first formulate this problem as a time-indexed ILP with valid inequalities. Three practical methods are then proposed to solve the problem: (i) an LP-based list scheduling algorithm, (ii) a single to multi-crew repair schedule conversion algorithm, and (iii) a dispatch rule based on $\rho$-factors which can be interpreted as Component Importance Measures. We show that the first two algorithms are $2$ and $\left(2 - \frac{1}{m}\right)$ approximations respectively. We also prove that the latter two algorithms are equivalent. Numerical results validate the effectiveness of the proposed methods.

\section*{Keywords}

Electricity distribution network, Natural disasters, Infrastructure resilience, Scheduling with soft precedence constraints, Time-indexed integer programming, LP-based list scheduling, Conversion algorithm, Component Importance Measure (CIM)

\section{Introduction}

Natural disasters, such as Hurricane Sandy in November 2012, the Christchurch Earthquake of February 2011 or the June 2012 Mid-Atlantic and Midwest Derecho, caused major damage to the electricity distribution networks and deprived homes and businesses of electricity for prolonged periods. Such power outages carry heavy social and economic costs. Estimates of the annual cost of power outages caused by severe weather between 2003 and 2012 range from \$18 billion to \$33 billion on average\cite{eotp2013resiliency}.  Physical damage to grid components must be repaired before power can be restored\cite{gridwise2013resilience,nerc2014sandy}. Hurricanes often cause storm surges that flood substations and corrode metal, electrical components and wiring \cite{newyork2013resilient}. Earthquakes can trigger ground liquefaction that damage buried cables and dislodge transformers \cite{eidinger2012christchurch}. Wind and ice storms bring down trees, breaking overhead cables and utility poles \cite{doe2012derecho}. As the duration of an outage increases, its economic and social costs rise exponentially. See \cite{wangresearch} \cite{Reed:2009ekbaca} for discussions of the impacts of natural disasters on power grids and \cite{Vugrin:2011kgba} \cite{Cimellaro:2010dpbaca} for its impact on other infrastructures.\\
It is important to distinguish the distribution repair and restoration problem discussed in this paper from the blackout restoration problem and the service restoration problem. Blackouts are large scale power outages (such as the 2003 Northeast US and Canada blackout) caused by an instability in the power generation and the high voltage transmission systems. This instability is triggered by an electrical fault or failure and is amplified by a cascade of component disconnections. Restoring power in the aftermath of a blackout is a different scheduling problem because most system components are not damaged and only need to be re-energized. See \cite{adibi1994power} \cite{adibi2006overcoming}
for a discussion of the blackout restoration problem and \cite{sun2011optimal} for a mixed-integer programming approach for solving this problem.  On the other hand, service restoration focuses on re-energizing a part of the local, low voltage distribution grid that has been automatically disconnected following a fault on a single component or a very small number of components. This can usually be done by isolating the faulted components and re-energizing the healthy parts of the network using switching actions. The service restoration problem thus involves finding the optimal set of switching actions. The repair of the faulted component is usually assumed to be taking place at a later time and is not considered in the optimization model. Several approaches have been proposed for the optimization of service restoration such as heuristics \cite{toune2002comparative} \cite{hou2011computation}, knowledge based systems \cite{ma1992operational}, and dynamic programming \cite{perez2008optimal}.

Unlike the outages caused by system instabilities or localized faults, outages caused by natural disasters require the repair of numerous components in the distribution grid before consumers can be reconnected. The research described in this paper therefore aims to schedule the repair of a significant number of damaged components, so that the distribution network can be progressively re-energized in a way that minimizes the cumulative harm over the total restoration horizon. Fast algorithms are needed to solve this problem because it must be solved immediately after the disaster and may need to be re-solved multiple times as more detailed information about the damage becomes available. Relatively few papers address this problem. Coffrin and Van Hentenryck \cite{coffrin2014transmission} propose an MILP formulation to co-optimize the sequence of repairs, the load pick-ups and the generation dispatch. However, the sequencing of repair does not consider the fact that more than one repair crew could work at the same time. Nurre et al. \cite{nurre2012restoring} formulate an integrated network design and scheduling (INDS) problem with multiple crews, which focuses on selecting a set of nodes and edges for installation in general infrastructure systems and scheduling them on work groups. They also propose a heuristic dispatch rule based on network flows and scheduling theory.

The rest of the paper is organized as follows. In Section 2, we define the problem of optimally scheduling multiple repair crews in a radial electricity distribution network after a natural disaster, and show that this problem is at least strongly $\mathcal{NP}$-hard. In Section 3, we formulate the post-disaster repair problem as an integer linear programming (ILP) using a network flow model and present a set of valid inequalities. Subsequently, we propose three polynomial time approximation algorithms based on adaptations of known algorithms in parallel machine scheduling theory, and provide performance bounds on their worst-case performance. A list scheduling algorithm based on an LP relaxation of the ILP model is discussed in Section 4; an algorithm which converts the optimal single crew repair sequence to a multi-crew repair sequence is presented in Section 5, along with a equivalent heuristic dispatch rule based on $\rho$-factors. In Section 6, we apply these methods to several standard test models of distribution networks. Section 7 draws conclusions.
\section{Problem formulation}
\label{sec:problemForm}
%
A distribution network can be represented by a graph $G$ with the set of nodes $N$ and the set of edges (a.k.a, lines) $L$. We assume that the network topology $G$ is radial, which is a valid assumption for most electricity distribution networks. Let $S \subset N$ represent the set of source nodes which are initially energized and $D = N \setminus S$ represent the set of sink nodes where consumers are located. An edge in $G$ represents a distribution feeder or some other connecting component. Severe weather can damage these components, resulting in a widespread disruption of power supply to the consumers. Let $L^{D}$ and $L^{I} = L \setminus L^D$ denote the sets of damaged and intact edges, respectively. Each damaged line $l \in L^D$ requires a repair time $p_{l}$ which is determined by the extent of damage and the location of $l$. We assume that it would take every crew the same amount of time to repair the same damaged line. Without any loss of generality, we assume that there is only one source node in $G$. If an edge is damaged, all downstream nodes lose power due to lack of electrical connectivity. In this paper, we consider the case where multiple crews work simultaneously and independently on the repair of separate lines, along with the special case where a single crew must carry all the repairs.  Finally, we make the assumption that crew travel times between damage sites are minimal and can be either ignored or factored into the component repair times. Therefore, our goal is to find a schedule by which the damaged lines should be repaired such that the aggregate harm due to loss of electric energy is minimized. We define this harm as follows:
%
\begin{equation}
\sum_{n \in N}w_{n}T_{n},
\label{eqn:harm}
\end{equation}
where $w_{n}$ is a positive quantity that captures the importance of node $n$ and $T_{n}$ is the time required to restore power at node $n$. 
The importance of a node can depend on multiple factors, including but not limited to, the amount of load connected to it, the type of load served, and interdependency with other critical infrastructure networks. For example, re-energizing a node supplying a major hospital should receive a higher priority than a node supplying a similar amount of residential load. Similarly, it is conceivable that a node that provides electricity to a water sanitation plant would be assigned a higher priority. These priority factors would need to be assigned by the utility companies and their determination is outside the scope of this paper. We simply assume knowledge of the $w_n$'s in the context of this paper.

The time to restore node $n$, $T_n$, is approximated by the \emph{energization time} $E_n$, which is defined as the time node $n$ first connects to the source node. System operators normally need to consider voltage and stability issues before restoring a node. However, this is not a major issue in distribution networks. The operators progressively restore a radial distribution network with enough generation capacity back to its normal operating state. And even if a rigorous power flow model is taken into account, the actual demands after re-energization are not known and hard to forecast. As a result, we model network connectivity using a simple network flow model, i.e., as long as a sink node is connected to the source, we assume that all the load on this node can be supplied without violating any security constraint. For simplicity, we treat the three-phase distribution network as if it were a single phase system. Our analysis could be extended to a three-phase system using a multi-commodity flow model, as in~\cite{yamangil2014designing}.


\subsection{Soft Precedence Constraints}
\label{subsec:soft}
We construct two simplified directed radial graphs to model the effect that the topology of the distribution network has on scheduling. The first graph, $G^{\prime}$, is called the `damaged component graph'. All nodes in $G$ that are connected by intact edges are contracted into a supernode in $G^{\prime}$. The set of edges in $G^{\prime}$ is the set of damaged lines in $G$, $L^D$. From a computational standpoint, the nodes of $G^{\prime}$ can be obtained by treating the edges in $G$ as undirected, deleting the damaged edges/lines, and finding all the connected components of the resulting graph. The set of nodes in each such connected component represents a (super)node in $G^{\prime}$. The edges in $G^{\prime}$ can then be placed straightforwardly by keeping track of which nodes in $G$ are mapped to a particular node in $G^{\prime}$. The directions to these edges follow trivially from the network topology. $G^{\prime}$ is useful in the ILP formulation introduced in Section~\ref{sec:ILP}.

The second graph, $P$, is called a `soft precedence constraint graph', which is constructed as follows. The nodes in this graph are the damaged lines in $G$ and an edge exists between two nodes in this graph if they share the same node in $G^{\prime}$. Computationally, the precedence constraints embodied in $P$ can be obtained by replacing lines in $G^{\prime}$ with nodes and the nodes in $G^{\prime}$ with lines. Such a graph enables us to consider the hierarchal relationship between damaged lines, which we define as \emph{soft precedence constraints}.

\begin{figure}[htbp]
\centering
\includegraphics[width=0.5\linewidth]{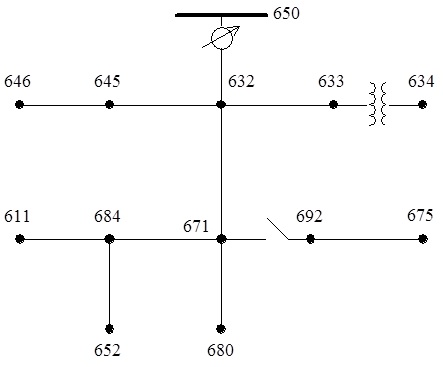}
\caption{IEEE 13 Node Test Feeder}
\label{fig:ieee13}
\end{figure}
\begin{figure}
\centering
\subfloat[$G'$ graph]{\includegraphics[width=0.4\textwidth]{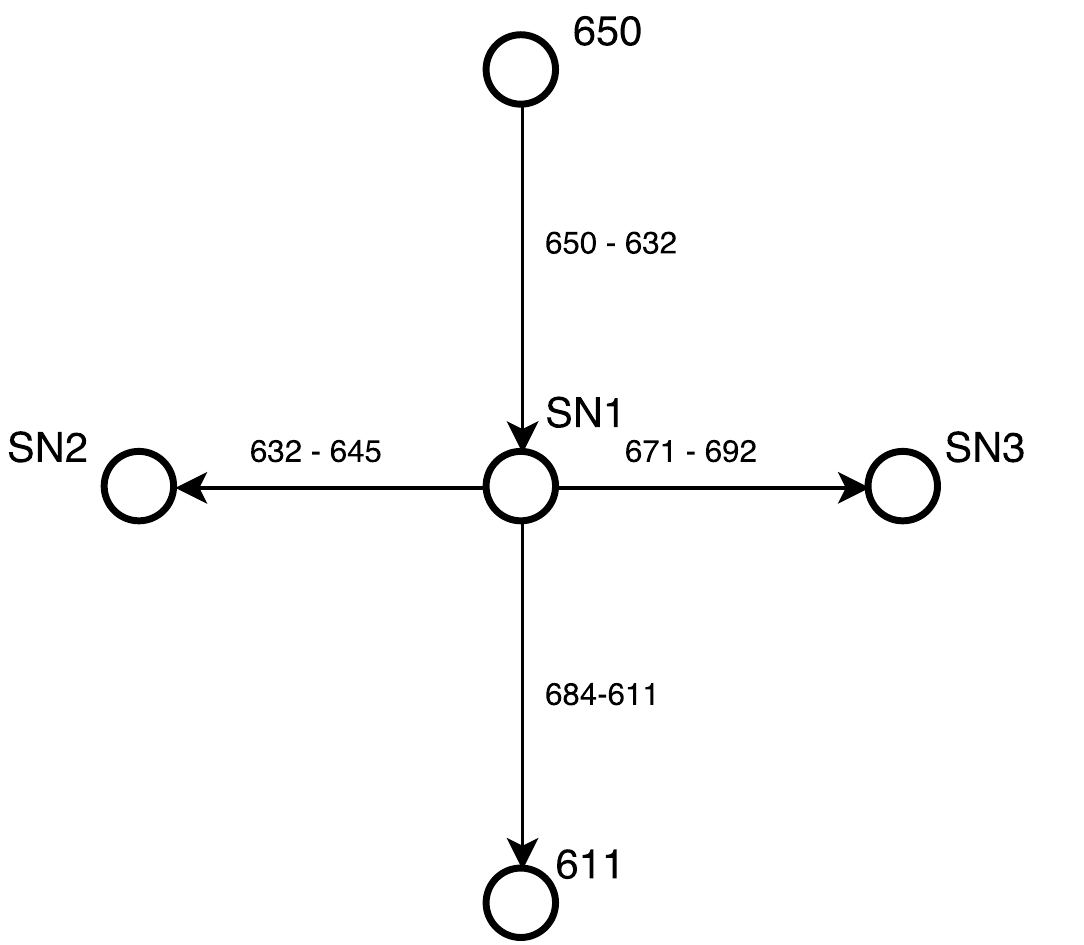}} \quad
\subfloat[$P$ graph]{\includegraphics[width=0.4\textwidth]{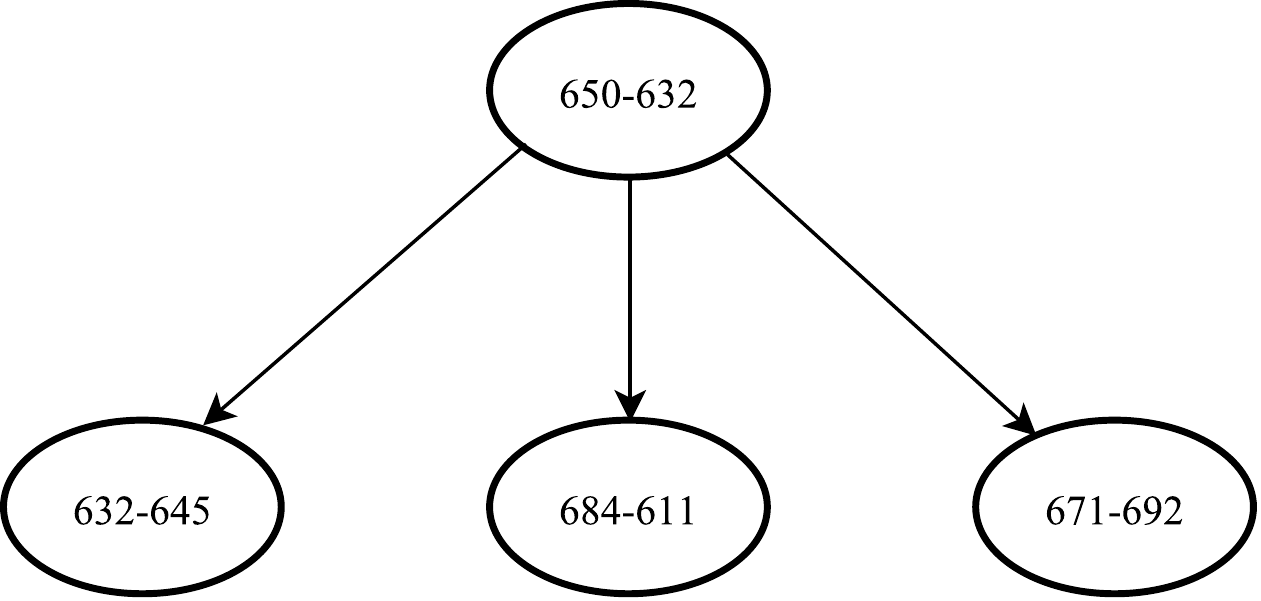}}
\caption{(a) The damaged component graph, $G^{\prime}$, obtained from Fig.~\ref{fig:ieee13}, assuming that the damaged lines are $650-632$, $632-645$, $684-611$ and $671-692$. (b) The corresponding soft precedence graph, $P$.}
\label{fig:sim_dag}
\end{figure}

A substantial body of research exists on scheduling with precedence constraints. In general, the precedence constraint $i \prec j$ requires that job $i$ be completed before job $j$ is started, or equivalently, $C_j \geq C_i$, where $C_j$ is the completion time of job $j$. Such precedence constraints, however, are not applicable in post-disaster restoration. While it is true that a sink node in an electrical network cannot be energized unless there is an intact path  (i.e., all damaged lines along that path have already been repaired) from the source (feeder) to this sink node, this does not mean that multiple lines on some path from the source to the sink cannot be repaired concurrently.

We keep track of two separate time vectors: the completion times of line repairs, denoted by $C_l$'s, and the energization times of nodes, denoted by $E_n$'s. While we have so far associated the term `energization time' with nodes in the given network topology, $G$, it is also possible to define energization times on the lines. Consider the example in Fig.~\ref{fig:sim_dag}. The precedence graph, $P$, requires that the line $650-632$ be repaired prior to the line $671-692$. If this (soft) precedence constraint is met, as soon as the line $671-692$ is repaired, it can be energized, or equivalently, all nodes in $\text{SN}_3$ (nodes $692$ and $675$) in the damaged component graph, $G^{\prime}$, can be deemed to be energized. The energization time of the line $671-692$ is therefore identical to the energization times of nodes $692$ and $675$. Before generalizing the above example, we need to define some notations. Given a directed edge $l$, let $h(l)$ and $t(l)$ denote the head and tail node of $l$. Let $l = h(l) \rightarrow t(l)$ be any edge in the damaged component graph $G^{\prime}$. Provided the soft precedence constraints are met, it is easy to see that $E_l = E_{t(l)}$, where $E_l$ is the energization time of line $l$ and $E_{t(l)}$ is the energization time of the node $t(l)$ in $G$. Analogously, the weight of node $t(l)$, $w_{t(l)}$, can be interpreted as a weight on the line $l$, $w_l$.
The soft precedence constraint, $i \prec_S j$, therefore implies that line $j$ cannot be energized unless line $i$ is energized, or equivalently, $E_j \geq E_i$, where $E_j$ is the energization time of line $j$.

\begin{proposition}
Given any feasible schedule of post-disaster repairs, the energization time $E_j$ always satisfies,
\begin{align}
E_j = \underset{i \, \preceq_S \, j}{\max} \,\; C_i
\end{align}
\end{proposition}
So far, we have modeled the problem of scheduling post-disaster repairs in electricity distribution networks as a parallel machine scheduling with \emph{outtree soft precedence constraints} in order to minimize the total weighted energization time, or equivalently, $P \vert outtree \text{ } soft \text{ } prec \vert \sum w_j E_j$, following Graham's notation in~\cite{graham1979optimization}.
\subsection{Complexity Analysis}
In this section, we study the complexity of the scheduling problem $P \vert outtree \text{ } soft \text{ } prec \vert \sum w_j E_j$ and show that it is at least strongly $\mathcal{NP}$-hard.
\begin{theorem}
The problem of scheduling post-disaster repairs in electricity distribution networks is at least strongly $\mathcal{NP}$-hard.
\end{theorem}
\begin{proof}
We show this problem is at least strongly $\mathcal{NP}$-hard using a reduction from the well-known identical parallel machine scheduling problem $P \ \vert \vert \ \sum w_{j}C_{j}$ defined as follows,

$P \ \vert \vert \ \sum w_{j}C_{j}$: Given a set of jobs $J$ in which $j$ has processing time $p_j$ and weight $w_j$, find a parallel machine schedule that minimizes the total weighted completion time $\sum w_{j}C_{j}$, where $C_j$ is the time when job $j$ finishes. $P \ \vert \vert \ \sum w_{j}C_{j}$ is strongly $\mathcal{NP}$-hard~\cite{pinedo2012scheduling, brucker2007scheduling}.

Given an instance of $P \ \vert \vert \ \sum w_{j}C_{j}$ defined as above, construct a star network $G_S$ with a source and $\vert J \vert$ sinks. Each sink $j$ has a weight $w_j$ and the line between the source and sink $j$ has a repair time of $p_j$. Whenever a line is repaired, the corresponding sink can be energized. Therefore the energization time of sink $j$ is equal to the completion time of line $j$. If one could solve the problem of scheduling post-disaster repairs in electricity distribution networks to optimality, then one can solve the problem in $G_S$ optimally and equivalently solve $P \ \vert \vert \ \sum w_{j}C_
{j}$.
\end{proof}
\section{Integer Linear Programming (ILP) formulation}
\label{sec:ILP}
With an additional assumption in this section that all repair times are integers, we model the post-disaster repair scheduling problem using time-indexed decision variables (see ), $x_{l}^{t}$, where $x_{l}^{t} = 1$ if line $l$ is being repaired by a crew at time period $t$. Variable $y_{l}^{t}$ denotes the repair status of line $l$ where $y_{l}^{t} = 1$ if the repair is done by the end of time period $t - 1$ and ready to energize at time period $t$. Finally, $u_{i}^{t} = 1$ if node $i$ is energized at time period $t$. Let $T$ denote the time horizon for the restoration efforts. Although we cannot know $T$ exactly until the problem is solved, a conservative estimate should work. Since $T_{i} = \sum_{t=1}^{T} (1 - u_{i}^{t})$ by discretization, the objective function of eqn.~\ref{eqn:harm} can be rewritten as:
\begin{equation}
\text{minimize} \ \ \ \sum_{t=1}^{T}\sum_{i \in N}w_{i}(1 - u_{i}^{t})
\label{eqn:obj}
\end{equation}
This problem is to be solved subject to two sets of constraints: (i) repair constraints and (ii) network flow constraints, which are discussed next. We mention in passing that the above time-indexed (ILP) formulation provides a strong relaxation of the original problem \cite{nurre2012restoring} and allows for modeling of different scheduling objectives without changing the structure of the model and the underlying algorithm.
\subsection{Repair Constraints}
Repair constraints model the behavior of repair crews and how they affect the status of the damaged lines and the sink nodes that must be re-energized. The three constraints below are used to initialize the binary status variables $y_{l}^{t}$ and $u_{i}^{t}$. Eqn.~\ref{eqn:inityd} forces $y_{l}^{t}=0$ for all lines which are damaged initially (i.e., at time $t=0$) while eqn. \ref{eqn:inityi} sets $y_{l}^{t}=1$ for all lines which are intact. Eqn.~\ref{eqn:initus} forces the status of all source nodes, which are initially energized, to be equal to $1$ for all time periods.
\begin{alignat}{2}
  y_{l}^{1} = 0, \quad &\forall l \in L^{D}
\label{eqn:inityd} \\
  y_{l}^{t} = 1, \quad &\forall l \in L^{I}, \  \forall t \in [1,T] \label{eqn:inityi} \\
  u_{i}^{t} = 1, \quad &\forall i \in S, \ \ \forall t \in [1,T] \label{eqn:initus}
\end{alignat}
where $T$ is the restoration time horizon. The next set of constraints is associated with the binary variables $x_{l}^{t}$. Eqn.~\ref{eqn:work} constrains the maximum number of crews working on damaged lines at any time period $t$ to be equal to $m$, where $m$ is the number of crews available.
\begin{equation}
\sum_{l \in L^D}x_{l}^{t} \leq m, \ \forall t \in [1,T]
\label{eqn:work}
\end{equation}
Observe that, compared to the formulation in \cite{nurre2012restoring}, there are no crew indices in our model. Since these indices are completely arbitrary, the number of feasible solutions can increase in crew indexed formulations, leading to enhanced computation time. For example, consider the simple network $i \rightarrow j \rightarrow k \rightarrow l$, where node $i$ is the source and all edges require a repair time of $5$ time units. If $2$ crews are available, suppose the optimal repair schedule is: `assign team $1$ to $i \rightarrow j$ at time $t=0$, team $2$ to $j \rightarrow k$ at $t=0$, and team $1$ to $k \rightarrow l$' at $t=5$.  Clearly, one possible equivalent solution conveying the same repair schedule and yielding the same cost, is: `assign team $2$ to $i \rightarrow j$ at $t=0$, team $1$ to $j \rightarrow k$ at $t=0$, and team $1$ to $k \rightarrow l$ at $t=5$'. In general, formulations without explicit crew indices may lead to a reduction in the size of the feasible solution set.
%
Although the optimal repair sequences obtained from such formulations do not natively produce the work assignments to the different crews, this is not an issue in practice because operators can choose to let a crew work on a line until the job is complete and assign the next repair job in the sequence to the next available crew (the first $m$ jobs in the optimal repair schedule can be assigned arbitrarily to the $m$ crews).
%
%

Finally, constraint eqn.~\ref{eqn:xy} formalizes the relationship between variables $x_{l}^{t}$ and $y_{l}^{t}$.
%
It mandates that $y_l^t$ cannot be set to $1$ unless at least $p_l$ number of $x_l^{\tau}$'s, $\tau \in [1, t - 1]$, are equal to $1$, where $p_l$ is the repair time of line $l$.
\begin{equation}
y_{l}^{t} \leq \frac{1}{p_{l}}\sum_{\tau=1}^{t-1} x_{l}^{\tau}, \ \forall l \in L^D, \ \forall t \in [1,T]
\label{eqn:xy}
\end{equation}
While we do not explicitly require that a crew may not leave its current job unfinished and take up a different job, it is obvious that such a scenario cannot be part of an optimal repair schedule.
%
%
\subsection{Network flow constraints}
\label{sec_netFlowConstraints}
%
We use a modified form of standard flow equations to simplify power flow constraints. Specifically, we require that the flows, originating from the source nodes (eqn.~\ref{eqn:source}), travel through lines which have already been repaired (eqn.~\ref{eqn:line}). Once a sink node receives a flow, it can be energized (eqn.~\ref{eqn:energize}).
%
\begin{alignat}{3}
  &\sum_{l \in \delta_{G}^{-}(i)} f_{l}^{t} \geq 0, \quad &\forall t \in [1,T], \ &\forall i \in S \label{eqn:source} \\
  &-M \times y_{l}^{t} \leq f_{l}^{t} \leq M \times y_{l}^{t}, \quad &\forall t \in [1,T], \ &\forall l \in L \label{eqn:line} \\
  %
  %
  &u_{i}^{t} \leq \sum_{l \in \delta_{G}^{+}(i)} f_{l}^t - \sum_{l \in \delta_{G}^{-}(i)} f_{l}^t, \quad &\forall t \in [1,T], &\forall i \in D
\label{eqn:energize}
\end{alignat}
%
%
%
%
In eqn. \ref{eqn:line}, $M$ is a suitably large constant, which, in practice, can be set equal to the number of sink nodes, $M = \vert D \vert$. In eqn. \ref{eqn:energize}, $\delta_{G}^{+}(i)$ and $\delta_{G}^{-}(i)$ denote the sets of lines on which power flows into and out of node $i$ in $G$ respectively.
\subsection{Valid inequalities}
Valid inequalities typically reduce the computing time and strengthen the bounds provided by the LP relaxation of an ILP formulation. We present the following shortest repair time path inequalities, which resemble the ones in \cite{nurre2012restoring}. A node $i$ cannot be energized until all the lines between the source $s$ and node $i$ are repaired. Since the lower bound to finish all the associated repairs is $\left \lfloor{SRTP_{i}/m}\right \rfloor$, where $m$ denotes the number of crews available and $SRTP_{i}$ denotes the shortest repair time path between $s$ and $i$, the following inequality is valid:
\begin{equation}
\sum_{t=1}^{\left \lfloor{SRTP_{i}/m}\right \rfloor - 1} u_{i}^{t} = 0, \quad \forall i \in N
\label{eqn:validineq}
\end{equation}
To summarize, the multi-crew distribution system post-disaster repair problem can be formulated as:
\begin{align}
\underset{}{\text{minimize}} \quad &\text{eqn.}~\ref{eqn:obj} \nonumber\\
\text{subject to} \quad
& \text{eqns.}~\ref{eqn:inityd} \ \scriptsize{\sim} \ \ref{eqn:validineq}
\end{align}
%
%
%
%
\section{List scheduling algorithms based on linear relaxation}
A majority of the approximation algorithms used for scheduling is derived from linear relaxations of ILP models, based on the scheduling polyhedra of completion vectors developed in \cite{queyranne1993structure} and \cite{schulz1996polytopes}. We briefly restate the definition of scheduling polyhedra and then introduce a linear relaxation based list scheduling algorithm followed by a worst case analysis of the algorithm.
\subsection{Linear relaxation of scheduling with soft precedence constraints}
\label{sec:LPrelax_softPrec}
%
%
A set of valid inequalities for $m$ identical parallel machine scheduling was presented in~\cite{schulz1996polytopes}:
\begin{align}
\sum_{j \in A} p_j C_j \geq f(A) := \frac{1}{2m}\left(\sum_{j \in A} p_j\right)^2 + \frac{1}{2} \, \sum_{j \in A} p_j^2 \,\;\quad \forall A \subset N
\label{eqn:validineqparallelschulz}
\end{align}
\begin{theorem}[\cite{schulz1996polytopes}]
The completion time vector $C$ of every feasible schedule on $m$ identical parallel machines satisfies inequalities~(\ref{eqn:validineqparallelschulz}).
\end{theorem}
The objective of the post-disaster repair and restoration is to minimize the harm, quantified as the total weighted energization time. With the previously defined soft precedence constraints and the valid inequalities for parallel machine scheduling, we propose the following LP relaxation:
\begin{align}
\underset{C, E}{\text{minimize}} \quad &\sum_{j \in L^D} w_j E_j \\
\text{subject to} \quad &C_j \geq p_j, \; \, \forall j \in L^D \label{eqn:cnstr1}\\
& E_j \geq C_j, \; \forall j \in L^D \label{eqn:cnstr2}\\
& E_j \geq E_i, \; \forall (i \rightarrow j) \in P \label{eqn:cnstr3}\\
& \sum_{j \in A} p_j C_j \geq \frac{1}{2m}\left(\sum_{j \in A} p_j\right)^2 + \frac{1}{2} \, \sum_{j \in A} p_j^2, \; \ \forall A \subset L^D \label{eqn:cnstr4}
\end{align}
where $P$ is the soft precedence graph discussed in Section~\ref{sec:problemForm} (see also Fig.~\ref{fig:sim_dag}). Eqn.~\ref{eqn:cnstr1} constrains the completion time of any damaged line to be lower bounded by its repair time, eqn.~\ref{eqn:cnstr2} ensures that any line cannot be energized until it has been repaired, eqn.~\ref{eqn:cnstr3} models the soft precedence constraints, and eqn.~\ref{eqn:cnstr4} characterizes the scheduling polyhedron.

The above formulation can be simplified by recognizing that the $C_j$'s are redundant intermediate variables. Combining eqns.~\ref{eqn:cnstr2} and ~\ref{eqn:cnstr4}, we have:
\begin{equation}
\sum_{j \in A} p_j E_j \ \geq \ \sum_{j \in A} p_j C_j \ \geq \ \frac{1}{2m} \left(\sum_{j \in A} p_j\right)^2 + \frac{1}{2} \, \sum_{j \in A} p_j^2, \quad \forall A \subset L^D
\end{equation}
which indicates that the vector of $E_j$'s satisfies the same valid inequalities as the vector of $C_j$'s. After some simple algebra, the LP-relaxation can be reduced to:
\begin{align}
\underset{E}{\text{minimize}} \quad &\sum_{j \in L^D} w_j E_j \\
\text{subject to} \quad
& E_j \geq p_j, \; \forall j \in L^D \label{eqn:cnstr21}\\
& E_j \geq E_i, \; \forall (i \rightarrow j) \in P \label{eqn:cnstr22}\\
& \sum_{j \in A} p_j E_j \geq \frac{1}{2m}\left(\sum_{j \in A} p_j\right)^2 + \frac{1}{2} \, \sum_{j \in A} p_j^2, \; \ \forall A \subset L^D \label{eqn:cnstr23}
\end{align}
We note that although there are exponentially many constraints in the above model, the separation problem for these inequalities can be solved in polynomial time using the ellipsoid method as shown in~\cite{schulz1996polytopes}.
\subsection{LP-based approximation algorithm}
List scheduling algorithms, which are among the simplest and most commonly used approximate solution methods for parallel machine scheduling problems \cite{queyranne2006approximation}, assign the job at the top of a priority list to whichever machine is idle first. An LP relaxation provides a good insight into the priorities of jobs and has been widely applied to scheduling with hard precedence constraints. We adopt a similar approach in this paper. Algorithm \ref{alg:LP_listSched}, based on a sorted list of the LP midpoints, summarizes our proposed approach. We now develop an approximation bound for Algorithm \ref{alg:LP_listSched}.
\begin{algorithm}
\caption{Algorithm for single/multiple crew repair scheduling in distribution networks, based on LP midpoints}
\label{alg:LP_listSched}
\begin{algorithmic}
\STATE \textit{Let $E^{LP}$ denote any feasible solution to the constraint eqns. \ref{eqn:cnstr21} - \ref{eqn:cnstr23}. Define the LP mid points to be $M_{j}^{LP} := E_{j}^{LP} - p_{j}/2, \ \forall j \in L^D$. Create a job priority list by sorting the $M_{j}^{LP}$'s in an ascending order. Whenever a crew is free, assign to it the next job from the priority list. The first $m$ jobs in the list are assigned arbitrarily to the $m$ crews. }
\end{algorithmic}
\end{algorithm}
%
%
\begin{proposition}
Let $E_j^H$ denote the energization time respectively of line $j$ in the schedule constructed by Algorithm \ref{alg:LP_listSched}. Then the following must hold,
\begin{align}
E_j^{H} \leq 2 E_j^{LP}, \quad \forall j \in L^D \label{eqn:ET}
\end{align}
\label{prop_firstLPSched}
\end{proposition}
\begin{proof}
Let $S_j^H$, $C_j^H$ and $E_j^H$ denote the start time, completion time respectively of some line $j$ in the schedule constructed by Algorithm \ref{alg:LP_listSched}.
Define $M := \left[M_j^{LP}: j = 1,2, \dots , \left\vert L^D \right\vert\right]$. Let $\tilde{M}$ denote $M$ sorted in ascending order, $\tilde{\mathcal{I}}_j$ denote the position of some line $j \in L^D$ in $\tilde{M}$, and $\left\{ k: \tilde{\mathcal{I}}_k \leq \tilde{\mathcal{I}}_j, k \neq j \right\} := R$ denote the set of jobs whose LP midpoints are upper bounded by $M_j^{LP}$. First, we claim that $S_j^{H} \leq \frac{1}{m} \, \sum_{i \in R} p_i$. To see why, split the set $R$ into $m$ subsets, corresponding to the schedules of the $m$ crews, i.e., $R = \bigcup_{k=1}^{m}R^k$. Since job $j$ is assigned to the first idle crew and repairs commence immediately, we have:
%
\begin{align}
  S_j^H = \text{min} \left\{\sum_{i \in R^k} p_i: k = 1,2, \dots m \right\} \leq\frac{1}{m} \, \sum_{k=1}^{m} \sum_{i \in R^k} p_i = \frac{1}{m} \, \sum_{i \in R} p_i \, ,
  \label{lemma4_impeq1}
\end{align}
where the inequality follows from the fact that the minimum of a set of positive numbers is upper bounded by the mean. Next, noting that $M_{j}^{LP} = E_{j}^{LP} - p_{j}/2$, we rewrite eqn.~\ref{eqn:cnstr23} as follows:
\begin{equation}
\sum_{j \in A} p_j M_j^{LP} \geq \frac{1}{2m}\left(\sum_{j \in A} p_j\right)^2, \; \forall A \subset L^D
\label{lemma4_impeq2}
\end{equation}
%
Now, letting $A = R$,  we have:
\begin{equation}
\left(\sum_{i \in R} p_i\right) M_j^{LP} \ \geq \ \sum_{i \in R} p_i M_i^{LP} \ \geq \ \frac{1}{2m}\left(\sum_{i \in R} p_i\right)^2 \, , 
\label{lemma4_impeq3}
\end{equation}
where the first inequality follows from the fact that $M_j^{LP} \geq M_i^{LP}$ for any $i \in R$. Combining eqns. \ref{lemma4_impeq1} and \ref{lemma4_impeq3}, it follows that $S_j^{H} \leq 2 M_j^{LP}$. Consequently, $C_j^H = S_j^H + p_j \leq 2 M_j^{LP} + p_j = 2 E_j^{LP}$. Then,
\begin{align}
E_j^H = \underset{i \preceq_S j}{\max} \, C_i^H \leq \underset{i \preceq j}{\max} \, 2E_i^{LP} = 2 E_j^{LP} ,
\end{align}
where the last equality follows trivially from the definition of a soft precedence constraint.
\end{proof}
\begin{theorem}
Algorithm \ref{alg:LP_listSched} is a 2-approximation.
\end{theorem}
\begin{proof}
Let $E_j^{*}$ denote the energization time of line $j$ in the optimal schedule. Then, with $E_j^{LP}$ being the solution of the linear relaxation,
\begin{equation}
\sum_{j \in L^D} w_j E_j^{LP} \leq \sum_{j \in L^D} w_j E_j^{*}
\label{eqn:objapprox}
\end{equation}
Finally, from eqns.~\ref{eqn:ET} and \ref{eqn:objapprox}, we have:
\begin{equation}
\sum_{j \in L^D} w_j E_j^{H} \leq 2 \, \sum_{j \in L^D} w_j E_j^{LP} \leq 2 \, \sum_{j \in L^D} w_j E_j^{*}
\end{equation}
\end{proof}
\section{An algorithm for converting the optimal single crew repair sequence to a multi-crew schedule}
\label{sec_convAlgo}
%
%
In practice, many utilities schedule repairs using a priority list \cite{xu2007optimizing}, which leaves much scope for improvement. We analyze the repair and restoration process as it would be done with a single crew because this provides important insights into the general structure of the multi-crew scheduling problem. Subsequently, we provide an algorithm for converting the single crew repair sequence to a multi-crew schedule, which is inspired by similar previous work in~\cite{chekuri2004approximation}, and analyze its worst case performance. Finally, we develop a multi-crew dispatch rule and compare it with the current practices of FirstEnergy Group~\cite{FirstEnergypractice} and Edison Electric Institute~\cite{EEIpractice}.
%
\subsection{Single crew restoration in distribution networks}
We show that this problem is equivalent to $1 \mid outtree \mid \sum w_{j}C_{j}$, which stands for scheduling to minimize the total weighted completion time of $N$ jobs with a single machine under `outtree' precedence constraints.  Outtree precedence constraints require that each job may have at most one predecessor. Given the manner in which we derive the soft precedence (see Section~\ref{sec:problemForm}), it is easy to see that $P$ will indeed follow outtree precedence requirements, i.e. each node in $P$ will have at most one predecessor, as long as the network topology $G$ does not have any cycles. We will show by the following lemma that the soft precedence constraints degenerate to the precedence constraints with one repair team.
\begin{proposition}
Given one repair crew, the optimal schedule in a radial distribution system must follow outtree precedence constraints, the topology of which follows the soft precedence graph $P$.
\label{prop:outtree}
\end{proposition}
\begin{proof}
Given one repair crew, each schedule can be represented by a sequence of damaged lines. Let $i-j$ and $j-k$ be two damaged lines such that the node $(j,k)$ is the immediate successor of node $(i,j)$ in the soft precedence graph $P$. Let $\pi$ be the optimal sequence and $\pi^{\prime}$ another sequence derived from $\pi$ by swapping $i-j$ and $j-k$. Denote the energization times of nodes $j$ and $k$ in $\pi$ by $E_{j}$ and $E_{k}$ respectively. Similarly, let $E_{j}^{\prime}$ and $E_{k}^{\prime}$ denote the energization times of nodes $j$ and $k$ in $\pi^{\prime}$. Define $f := \sum_{n \in N}w_{n}E_{n}$.

Since node $k$ cannot be energized unless node $j$ is energized and until the line between it and its immediate predecessor is repaired, we have $E_{k}^{\prime} = E_{j}^{\prime}$ in $\pi'$ and $E_{k} > E_{j}$ in $\pi$. Comparing $\pi$ and $\pi^{\prime}$, we see that node $k$ is energized at the same time, i.e., $E_{k}^{\prime} = E_{k}$, and therefore, $E_{j}^{\prime} > E_{j}$. Thus:
\begin{equation}
\begin{aligned}
f(\pi^{\prime}) - f(\pi) & = (w_{j}E_{j}' + w_{k}E_{k}') - (w_{j}E_{j} + w_{k}E_{k}) \\
& = w_{j}(E_{j}' - E_{j}) + w_{k}(E_{k}' - E_{k}) > 0
\end{aligned}
\end{equation}
%
Therefore, any job swap that violates the outtree precedence constraints will strictly increase the objective function. Consequently, the optimal sequence must follow these constraints.
\end{proof}
 It follows immediately from Proposition~\ref{prop:outtree} that:
\begin{lemma}
Single crew repair and restoration scheduling in distribution networks is equivalent to $1 \mid outtree \mid \sum_j w_{j}C_{j}$, where the outtree precedences are given in the soft precedence constraint graph $P$.
\end{lemma}
\subsection{Recursive scheduling algorithm for single crew restoration scheduling}
As shown above, the single crew repair scheduling problem in distribution networks is equivalent to $1 \mid outtree \mid \sum w_{j}C_{j}$, for which an optimal algorithm exists~\cite{adolphson1973optimal}. We will briefly discuss this algorithm and the reasoning behind it. Details and proofs can be found in \cite{brucker2007scheduling}.
%
%
Let $J^D \subseteq L^D$ denote any subset of damaged lines. Define:
\begin{equation*}
w\left(J^D\right) := \sum_{j \in J^D} w_j,\quad p\left(J^D\right) := \sum_{j \in J^D} p_j, \quad q\left(J^D\right) := \frac{w\left(J^D\right)}{p\left(J^D\right)}
\end{equation*}
Algorithm \ref{alg:outree_merge}, adapted from \cite{brucker2007scheduling} with a change of notation, finds the optimal repair sequence by recursively merging the nodes in the soft precedence graph $P$. The input to this algorithm is the precedence graph $P$. Let $N(P) = \left\{1,2, \dots \vert N(P) \vert\right\}$ denote the set of nodes in $P$ (representing the set of damaged lines, $L^D$), with node $1$ being the designated root. The predecessor of any node $n \in P$ is denoted by $\text{pred}(n)$. Lines $1-7$ initialize different variables. In particular, we note that the predecessor of the root is arbitrarily initialized to be $0$ and its weight is initialized to $-\infty$ to ensure that the root node is the first job in the optimal repair sequence. Broadly speaking, at each iteration, a node $j \in N(P)$ ($j$ could also be a group of nodes) is chosen to be merged into its immediate predecessor $i \in N(P)$ if $q(j)$ is the largest. The algorithm terminates when all nodes have been merged into the root. Upon termination, the optimal single crew repair sequence can be recovered from the predecessor vector and the element $A(1)$, which indicates the last job finished.
%

We conclude this section by noting that Algorithm \ref{alg:outree_merge} requires the precedence graph $P$ to have a defined root. However, as illustrated in Section \ref{sec:problemForm}, it is quite possible for $P$ to be a forest, i.e., a set of disjoint trees. In such a situation, $P$ can be modified by introducing a dummy root node with a repair time of $0$ and inserting directed edges from this dummy root to the roots of each individual tree in the forest. This fictitious root will be the first job in the repair sequence returned by the algorithm, which can then be stripped off.
\begin{algorithm}
\caption{Optimal algorithm for single crew repair and restoration in distribution networks.}
\label{alg:outree_merge}
\begin{algorithmic}[1]
    \STATE $w(1) \leftarrow -\infty$; \quad $\text{pred}(1) \leftarrow 0$;
    %
    \FOR{$n = 1$ to $\vert N(P) \vert$}
    %
    \STATE $A(n) \leftarrow n; \quad B_n \leftarrow \{n\}; \quad q(n) \leftarrow w(n)/p(n)$;
    \ENDFOR
    \FOR{$n = 2$ to $\vert N(P) \vert$}
    \STATE $\text{pred}(n) \leftarrow \text{parent of $n$ in $P$}$;
    \ENDFOR
    \STATE $\text{nodeSet} \leftarrow \{1,2, \cdots, \vert N(P) \vert\}$;
    %
    \WHILE{$\text{nodeSet} \neq \{1\}$}
    %
    \STATE Find $j \in \text{nodeSet}$ such that $q(j)$ is largest; \quad \texttt{$\%$ ties can be broken arbitrarily}
    %
    %
    \STATE Find $i$ such that $\text{pred}(j) \in B_{i}$, $i = 1,2,\dots \vert N(P) \vert$;
    %
    \STATE $w(i) \leftarrow w(i) + w(j)$;
    %
    \STATE $p(i) \leftarrow p(i) + p(j)$;
    %
    \STATE $q(i) \leftarrow w(i)/p(i)$;
    %
    \STATE $\text{pred}(j) \leftarrow A(i)$;
    %
    \STATE $A(i) \leftarrow A(j)$;
    %
    \STATE $B_i \leftarrow \{B_i, B_j\}$; \quad \texttt{$\%$ `$,$' denotes concatenation}
    %
    \STATE $\text{nodeSet} \leftarrow \text{nodeSet} \setminus \{j\}$;
    \ENDWHILE
\end{algorithmic}
\end{algorithm}
%
%
\subsection{Conversion algorithm and an approximation bound}
A greedy procedure for converting the optimal single crew sequence to a multiple crew schedule is given in Algorithm~\ref{alg:convert_1_to_m}. We now prove that it is a $\left(2-\frac{1}{m}\right)$ approximation algorithm. We start with two lemmas that provide lower bounds on the minimal  harm for an $m$-crew schedule, in terms of the minimal harms for single crew and $\infty$-crew schedules. Let $H^{1,\ast}$, $H^{m,\ast}$ and $H^{\infty,\ast}$  denote the minimal harms when the number of repair crews is $1$, some arbitrary $m$ ($2 \leq m < \infty$), and $\infty$ respectively.
\begin{algorithm}
\caption{Algorithm for converting the optimal single crew schedule to an $m$-crew schedule}
\label{alg:convert_1_to_m}
\begin{algorithmic}
\STATE \textit{Treat the optimal single crew repair sequence as a priority list, and, whenever a crew is free, assign to it the next job from the list. The first $m$ jobs in the single crew repair sequence are assigned arbitrarily to the $m$ crews.}
\end{algorithmic}
\end{algorithm}

\begin{proposition}
$H^{m, *} \geq \frac{1}{m} \, H^{1, *}$
\label{prop:Hm1}
\end{proposition}
\begin{proof}
Given an arbitrary $m$-crew schedule $S^m$ with harm $H^m$, we first construct a $1$-crew repair sequence, $S^1$. We do so by sorting the energization times of the damaged lines in $S^m$ in ascending order and assigning the corresponding sorted sequence of lines to $S^1$. Ties, if any, are broken according to precedence constraints or arbitrarily if there is none. By construction, for any two damaged lines $i$ and $j$ with precedence constraint $i \prec j$, the completion time of line $i$ must be strictly smaller than the completion time of line $j$ in $S^1$, i.e., $C_i^1 < C_j^1$. Additionally, $C_i^1 = E_i^1$ because the completion and energization times of lines are identical for a $1$-crew repair sequence which also meets the precedence constraints of $P$.

Next, we claim that $E_i^1 \leq m E_i^m$, where $E_i^1$ and $E_i^m$ are the energization times of line $i$ in $S^1$ and $S^m$ respectively.
In order to prove it, we first observe that:
\begin{align}
E_i^1 = C_i^1 = \sum_{\{j: \ E_j^m \leq E_i^m\}} p_j \ \ \leq \sum_{\{j: \ C_j^m \leq E_i^m\}} p_j \, ,
\end{align}
where the second equality follows from the manner we constructed $S^1$ from $S^m$ and the inequality follows from the fact that $C_j^m \leq E_j^m \Rightarrow \{j: \ E_j^m \leq E_i^m\} \subseteq \{j: \ C_j^m \leq E_i^m\}$ for any $m$-crew schedule. In other words, the number of lines that have been energized before line $i$ is energized is a subset of the number of lines on which repairs have been completed before line $i$ is energized. Next, we split the set $\{j:C_j^m \leq E_i^m\} := R$ into $m$ subsets, corresponding to the schedules of the $m$ crews in $S^m$, i.e., $R = \bigcup_{k=1}^{m}R^k$, where $R^k$ is a subset of the jobs in $R$ that appear in the $k^{th}$ crew's schedule. It is obvious that the sum of the repair times of the lines in each $R^k$ can be no greater than $E_i^m$. Therefore,
\begin{align}
  E_i^1 \leq \sum_{\{j: \ C_j^m \leq E_i^m\}} p_j := \sum_{j \in R} p_j = \sum_{k=1}^m \left(\sum_{j \in R^k} p_j\right) \leq m E_i^m
\end{align}
Proceeding with the optimal $m$-crew schedule $S^{m,\ast}$ instead of an arbitrary one, it is easy to see that $E_i^1 \leq m E_i^{m,\ast}$, where $E_i^{m,\ast}$ is the energization time of line $i$ in $S^{m,\ast}$. The lemma then follows straightforwardly.
\begin{align}
H^{m, *} = \sum_{i \in L^D} w_i E_i^{m,*} \geq \sum_{i \in L^D} w_i \, \frac{1}{m} \, E_i^1 = \frac{1}{m} \, \sum_{i \in L^D} w_i \, E_i^1 = \frac{1}{m} \, H^1 \geq \frac{1}{m} \, H^{1, *}
\end{align}
\end{proof}
%
%
\begin{proposition}
$H^{m,\ast} \geq H^{\infty,\ast}$
\label{prop:Hminf}
\end{proposition}
\begin{proof}
This is intuitive, since the harm is minimized when the number of repair crews is at least equal to the number of damaged lines. In the $\infty$-crew case, every job can be assigned to one crew. For any damaged line $j \in L^D$, $C^{\infty}_j = p_j$ and $E^{\infty}_j = \max_{i \preceq j} C^{\infty}_i = \max_{i \preceq j} p_i$. Also, $C^{m,\ast}_j \geq p_j = C^{\infty}_j$ and $E^{m,\ast}_j = \max_{i \preceq j} C^{m,\ast}_i \geq \max_{i \preceq j} p_i = E^{\infty}_j$. Therefore:
\begin{align}
H^{m, *} = \sum_{j \in L^D} w_j E_j^{m,*} \geq \sum_{j \in L^D} w_j E^{\infty}_j = H^{\infty,\ast}
\end{align}
\end{proof}
\begin{proposition}
Let $E_j^m$ be the energization time of line $j$ after the conversion algorithm is applied to the optimal single crew repair schedule. Then, $\forall j \in L^D$, $E_{j}^{m} \leq \frac{1}{m} \, E_{j}^{1,\ast} + \frac{m-1}{m} \, E^{\infty,\ast}_j$. 
\label{prop:ejm}
\end{proposition}
\begin{proof}
Let $S_j^m$ and $C_J^m$  denote respectively the start and energization times of some line $j \in L^D$ in the $m$-crew repair schedule, $S^m$, obtained by applying the conversion algorithm to the optimal $1$-crew sequence, $S^{1,\ast}$. Also, let $\mathcal{I}_j$ denote the position of line $j$ in $S^{1,\ast}$ and $\{k: \mathcal{I}_k < \mathcal{I}_j\} := R$ denote the set of all lines completed before $j$ in $S^{1,\ast}$. First, we claim that: $S_j^m \leq \frac{1}{m} \, \sum_{i \in R} p_i$. A proof can be constructed by following the approach taken in the proof of Proposition \ref{prop_firstLPSched} and is therefore omitted. Now:
%
\begin{align}
C_j^m &= S_j^m + p_j \\
&\leq \frac{1}{m} \, \sum_{i \in R} p_i + p_j \\
&= \frac{1}{m} \, \sum_{i \in R \,  \cup \, j} p_i + \frac{m-1}{m} \, p_j \\
& = \frac{1}{m} \, C_j^{1,\ast} + \frac{m-1}{m} \, p_j
\end{align}
and
\begin{align}
E_j^m &= \max_{i \preceq_S j} C_i^m \\
&\leq \max_{i \preceq j} \frac{1}{m} \, C_i^{1,\ast} + \max_{i \preceq j} \frac{m-1}{m} \, p_i \\
&= \frac{1}{m} \, C_j^{1,\ast} + \frac{m-1}{m} \, \max_{i \preceq j} p_i \\
&= \frac{1}{m} \, E_{j}^{1,\ast} + \frac{m-1}{m} \, E^{\infty,\ast}_j
\end{align}
\end{proof}
\begin{theorem}
The conversion algorithm is a $\left(2 - \frac{1}{m}\right)$-approximation.
\end{theorem}
\begin{proof}
%
\begin{align}
H^m &= \sum_{j \in L^D} w_j E_j^{m} \\
&\leq \sum_{j \in L^D} w_j \left(\frac{1}{m} E_{j}^{1, *} + \frac{m-1}{m}E^{\infty}_j\right) \quad \textrm{$\cdots$ using Proposition~\ref{prop:ejm}} \\
&= \frac{1}{m} \, \sum_{j \in L^D} w_j E_{j}^{1, *} +  \frac{m-1}{m} \, \sum_{j \in L^D} w_j E^{\infty}_j \\
&= \frac{1}{m} \, H^{1, *} + \frac{m-1}{m} \, H^{\infty, *} \\
&\leq \frac{1}{m} \, \left(m H^{m, *}\right) + \frac{m-1}{m} \, H^{m, *} \quad \textrm{$\cdots$ using Propositions~\ref{prop:Hm1} - \ref{prop:Hminf}} \\
&= \left(2 - \frac{1}{m}\right) H^{m, *}
\end{align}
\end{proof}

\subsection{A Dispatch Rule}
\label{Sec:dispatchRule}
We now develop a multi-crew dispatch rule from a slightly different perspective, and show that it is equivalent to the conversion algorithm. In the process, we define a parameter, $\rho(l)$, $\forall l \in L^D$, which can be interpreted as a `component importance measure' (CIM) in the context of reliability engineering. This allows us to easily compare our conversion algorithm to standard utility practices. Towards that goal, we revisit the single crew repair problem, in conjunction with the algorithm proposed in~\cite{horn1972single}.

Let $S_{l}$ denote the set of all trees rooted at node $l$ in $P$ and $s^*_{l} \in S_{l}$ denote the minimal subtree which satisfies:
\begin{equation}
  \rho(l) := \frac{\sum_{j \in N(s^{\ast}_{l})} w_j}{\sum_{j \in N(s^{\ast}_{l})} p_j} = \max_{s_{l} \in S_{l}} \left( \frac{\sum_{j \in N(s_{l})} w_j}{\sum_{j \in N(s_{l})} p_j} \right) ,
\label{eqn:rhofacdev}
\end{equation}
where $N(s_{l})$ is the set of nodes in $s_{l}$. We define the ratio on the left-hand side of the equality in eqn.~\ref{eqn:rhofacdev} to be the $\rho$-factor of line $l$, denoted by $\rho(l)$. We refer to the tree $s^*_{l}$ as the minimal $\rho$-maximal tree rooted at $l$, which resembles the definitions discussed in~\cite{sidney1975decomposition}. With $\rho$-factors calculated for all damaged lines, the repair scheduling with single crew can be solved optimally, as stated in Algorithm \ref{alg:singlelst} below, adopted from~\cite{horn1972single}. Note that $\rho$-factors are defined based on the soft precedence graph $P$, whereas the following dispatch rules are stated in terms of the original network $G$ to be more in line with industry practices.
\begin{algorithm}
\caption{Algorithm for single crew repair scheduling in distribution networks}
\label{alg:singlelst}
\begin{algorithmic}
\STATE \textit{Whenever the crew is free, say at time $t$, select among the candidate lines the one with the highest $\rho$-factor. The candidate set comprises all the damaged lines, one of whose end points is within the set of energized nodes at time $t$.}
\end{algorithmic}
\end{algorithm}

It has been proven in~\cite{adolphson1973optimal} that Algorithms~\ref{alg:outree_merge} and~\ref{alg:singlelst} are equivalent. The $\rho$-factors can be calculated in multiple ways: (1) following the method proposed in~\cite{horn1972single}, (2) as a byproduct of Algorithm~\ref{alg:outree_merge}, and (3) using a more general method based on parametric minimum cuts in an associated directed precedence graph~\cite{lawler1978sequencing}. Algorithm~\ref{alg:singlelst} can be extended straightforwardly to accommodate multiple crews. However, in this case, it could happen that the number of damaged lines that are connected to energized nodes is smaller than the number of available repair crews. To cope with this issue, we also consider the lines which are connected to the lines currently being repaired, as described in Algorithm~\ref{alg:multilst} below.
\begin{algorithm}
\caption{Algorithm for multi-crew repair scheduling in distribution networks}
\label{alg:multilst}
\begin{algorithmic}
\STATE \textit{Whenever a crew is free, say at time $t$, select among the remaining candidate lines the one with the highest $\rho$-factor. The candidate set consists of all the damaged lines that are connected to already energized nodes, as well as the lines that are being repaired at time $t$. }
\end{algorithmic}
\end{algorithm}
\begin{theorem}
Algorithm \ref{alg:multilst} is equivalent to Algorithm \ref{alg:convert_1_to_m} discussed in Section \ref{sec_convAlgo}.
\label{thm:equiv5and3}
\end{theorem}
\begin{proof}
As stated above, Algorithms~\ref{alg:outree_merge} and~\ref{alg:singlelst} are both optimal algorithms and we assume that, without loss of generality, they produce the same optimal sequences. Then it suffices to show that Algorithm~\ref{alg:multilst} converts the sequence generated by Algorithm~\ref{alg:singlelst} in the same way that Algorithm~\ref{alg:convert_1_to_m} does to Algorithm~\ref{alg:outree_merge}.

The proof is by induction on the order of lines being selected. In iteration $1$, it is obvious that Algorithms~\ref{alg:singlelst} and \ref{alg:multilst} choose the same line for repair. Suppose this is also the case for iterations $2$ to $t-1$, with the lines chosen for repair being $l_1$, $l_2$, $l_3$, $\cdots$, and $l_{t-1}$ respectively. Then, in iteration $t$, the set of candidate lines for both algorithms is the set of immediate successors of the supernode \{$l_1$, $l_2$, $\cdots$, $l_{t-1}$\}. Both algorithms will choose the job with the largest $\rho$-factor in iteration $t$, thereby completing the induction process.
\end{proof}
\subsection{Comparison with current industry practices}
\label{Sec:currentpractice}

According to FirstEnergy Group~\cite{FirstEnergypractice}, repair crews will ``address outages that restore the largest number of customers before moving to more isolated problems''. This policy can be interpreted as a priority-based scheduling algorithm and fits within the scheme of the dispatch rule discussed above, the difference being that, instead of selecting the line with the largest $\rho$-factor, FirstEnergy chooses the one with the largest weight (which turns out to be the number of customers). Edison Electric Institute \cite{EEIpractice} states that crews are dispatched to ``repair lines that will return service to the largest number of customers in the least amount of time''. This policy is analogous to Smith's ratio rule~\cite{smith1956various} where jobs are sequenced in descending order of the ratios $w_l/p_l$, ensuring that jobs with a larger weight and a smaller repair time have a higher priority. The parameter, $\rho(l)$, can be viewed as a generalization of the ratio $w_l/p_l$ and characterizes the repair priority of some damaged line $l$ in terms of its own importance as well as the importance of its succeeding nodes in $P$. Stated differently, $\rho(l)$ can be interpreted as a broad component importance measure for line $l$. Intuitively, we expect a dispatch rule based on $\rho(l)$ to work better than current industry practice since it takes a more holistic view of the importance of a line and, additionally, has a proven theoretical performance bound. Simulation results presented later confirm that a dispatch rule based on our proposed $\rho$-factors indeed results in a better restoration trajectory compared to standard industry practices.
\section{Case Studies}
In this section, we apply our proposed methods to three IEEE standard test feeders of different sizes. We consider the worst case, where all lines are assumed to be damaged. In each case, the importance factor $w$ of each node is a random number between 0 and 1, with the  exception of a randomly selected extremely important node with $w = 5$. The repair times are uniformly distributed on integers from $1$ to $10$. We compare the performances of the three methods, with computational time being of critical concern since restoration activities, in the event of a disaster, typically need to be performed in real time or near real time. All experiments were performed on a desktop with a 3.10 GHz Intel Xeon processor and 16 GB RAM. The ILP formulation was solved using Julia for Mathematical Programming with Gurobi 6.0.
\subsection{IEEE 13-Node Test Feeder}
The first case study is performed on the IEEE 13 Node Test Feeder shown in Fig.~\ref{fig:ieee13}, assuming that the number of repair crews is $m=2$. Since this distribution network is small, an optimal solution could be obtained by solving the ILP model. We ran 1000 experiments in order to compare the performances of the two heuristic algorithms w.r.t the ILP formulation.

Fig.~\ref{fig:ieee13gap} shows the density plots of optimality gaps of LP-based list scheduling algorithm (LP) and the conversion algorithm (CA), along with the better solution from the two (EN). Fig.~\ref{fig:ieeegap_00} shows the optimality gaps when all repair times are integers. The density plot in this case is cut off at $0$ since the ILP solves the problem optimally. Non-integer repair times can be scaled up arbitrarily close to integer values, but at the cost of reduced computational efficiency of the ILP. Therefore, in the second case, we perturbed the integer valued repair times by $\pm 0.1$, which represents a reasonable compromise between computational accuracy and efficiency. The optimality gaps in this case are shown in Fig.~\ref{fig:ieeegap_01}. In this case, we solved the ILP using rounded off repair times, but the cost function was computed using the (sub-optimal) schedules provided by the ILP model and the actual non-integer repair times. This is why the heuristic algorithms sometimes outperform the ILP model, as is evident from Fig.~\ref{fig:ieeegap_01}. In both cases, the two heuristic algorithms can solve most of the instances with an optimality gap of  less than $10\%$. Comparing the two methods, we see that the conversion algorithm (CA) has a smaller mean optimality gap, a thinner tail, and a better worst case performance. However, this does not mean that the conversion algorithm is universally superior. In approximately $34\%$ of the problem instances, we have found that the LP-based list scheduling algorithm yields a solution which is no worse than the one provided by the conversion algorithm.
\begin{figure}[h]
\subfloat[Integer-valued repair time]{
\label{fig:ieeegap_00}
\begin{minipage}[t]{0.5\textwidth}
\centering
\includegraphics[width = 1\textwidth]{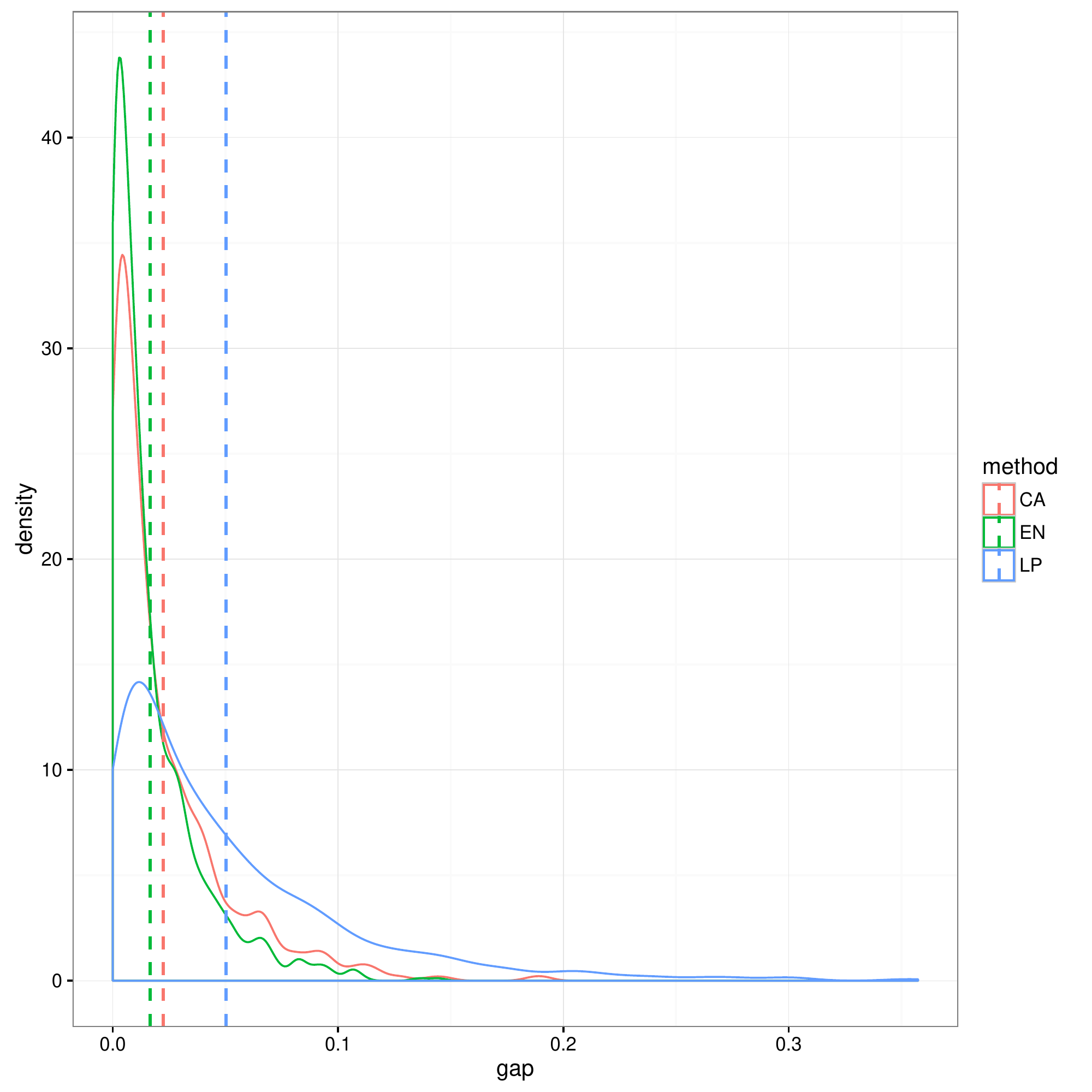}
\end{minipage}
}
\subfloat[Integer-$\pm 0.1$-valued repair time]{
\label{fig:ieeegap_01}
\begin{minipage}[t]{0.5\textwidth}
\centering
\includegraphics[width = 1\textwidth]{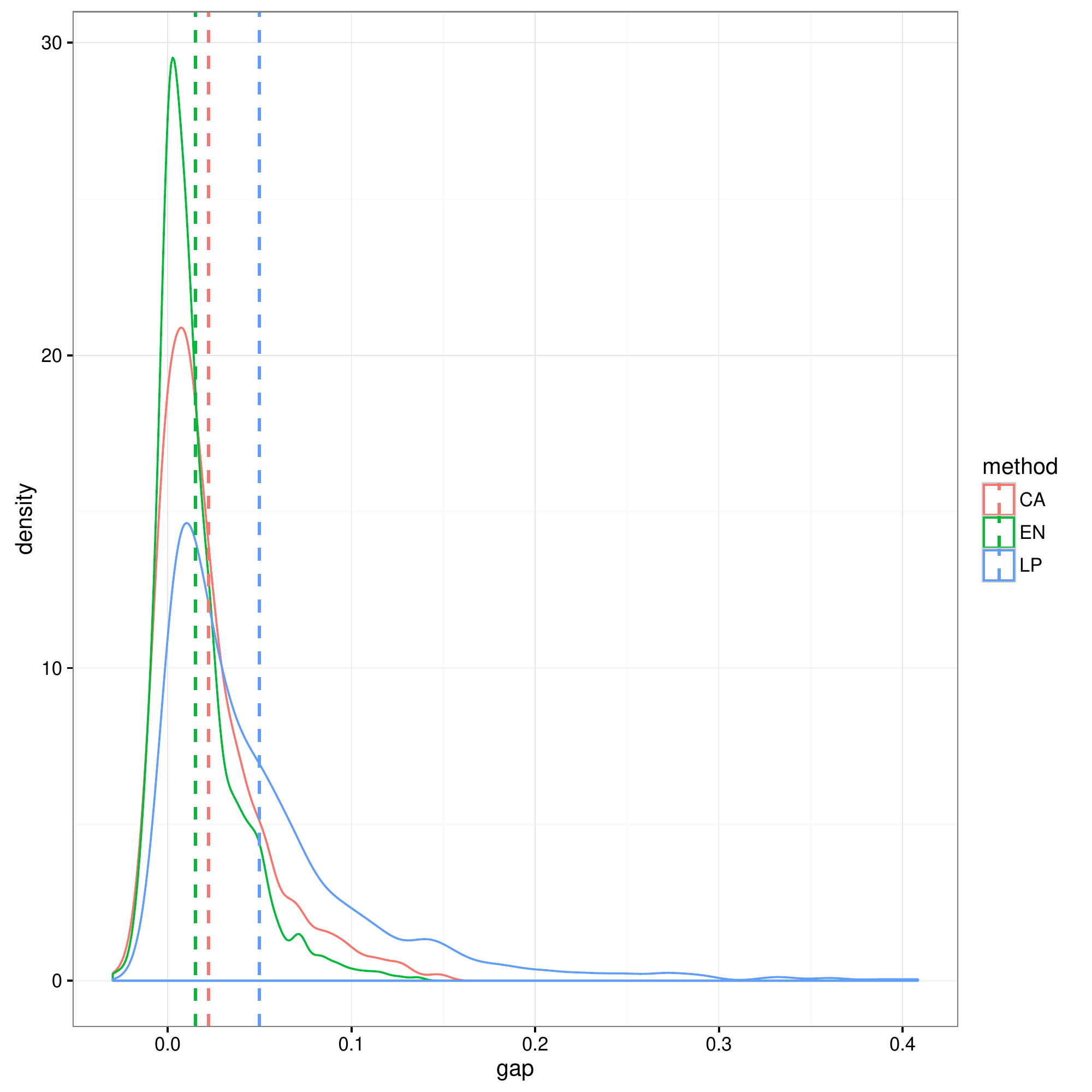}
\end{minipage}
}
\caption{Density plot of optimality gap with means}
\label{fig:ieee13gap}
\end{figure}
\subsection{IEEE 123-Node Test Feeder}
Next, we ran our algorithms on one instance of the IEEE 123-Node Test Feeder~\cite{kersting2001radial} with $m=5$. Since solving such problems to optimality using the ILP requires a prohibitively large computing time, we allocated a time budget of one hour. As shown in Table~\ref{tab:123comparison}, both LP and HA were able to find a better solution than the ILP, at a fraction of the computing time.

\begin{table}[htbp]
\centering
\begin{tabular}{c|c|c}
\hline
                     & Harm                 & Time(s)     \\ \hline
ILP                  & $3.0788 \times 10^3$ & 3600        \\ \hline
Conversion Algorithm & $2.2751 \times 10^3$ & \textless1s \\ \hline
Linear Relaxation    & $2.3127 \times 10^3$ & 24s         \\ \hline
\end{tabular}
\caption{Performance comparison for the IEEE 123-node test feeder}
\label{tab:123comparison}
\end{table}

\subsection{IEEE 8500-Node Test Feeder}
Finally, we tested the two heuristic algorithms on one instance of the IEEE 8500-Node Test Feeder medium voltage subsystem \cite{arritt2010the} containing roughly 2500 lines, with $m=10$. We did not attempt to solve the ILP model in this case. As shown in Table~\ref{tab:8500comparison}, it took about more than 60 hours to solve its linear relaxation (which is reasonable since we used the ellipsoid method to solve the LP with exponentially many constraints) and the conversion algorithm actually solved the instance in two and a half minutes.

We also compared the performance of our proposed $\rho$-factor based dispatch rule to standard industry practices discussed in Section~\ref{Sec:currentpractice}. We assign the same weights to nodes for all three dispatch rules. The plot of network functionality (fraction restored) as a function of time in Fig.~\ref{fig:traj_comparison} shows the comparison of functionality trajectories. While the time to full restoration is almost the same for all three approaches, it is clear that our proposed algorithm results in a greater network functionality at intermediate times. Specifically, an additional $10\%$ (approximately) of the network is restored approximately halfway through the restoration process, compared to standard industry practices.
%
\begin{table}[htbp]
\centering
\begin{tabular}{c|c|c}
\hline
                     & Harm                 & Time     \\ \hline
Conversion Algorithm & $7.201 \times 10^5$ & 150.64 s \\ \hline
Linear Relaxation    & $7.440 \times 10^5$ & 60.58 h         \\ \hline
\end{tabular}
\caption{Performance comparison for the IEEE 8500-node test feeder}
\label{tab:8500comparison}
\end{table}
\begin{figure}[htbp]
\centering
\includegraphics[width = 0.5\columnwidth]{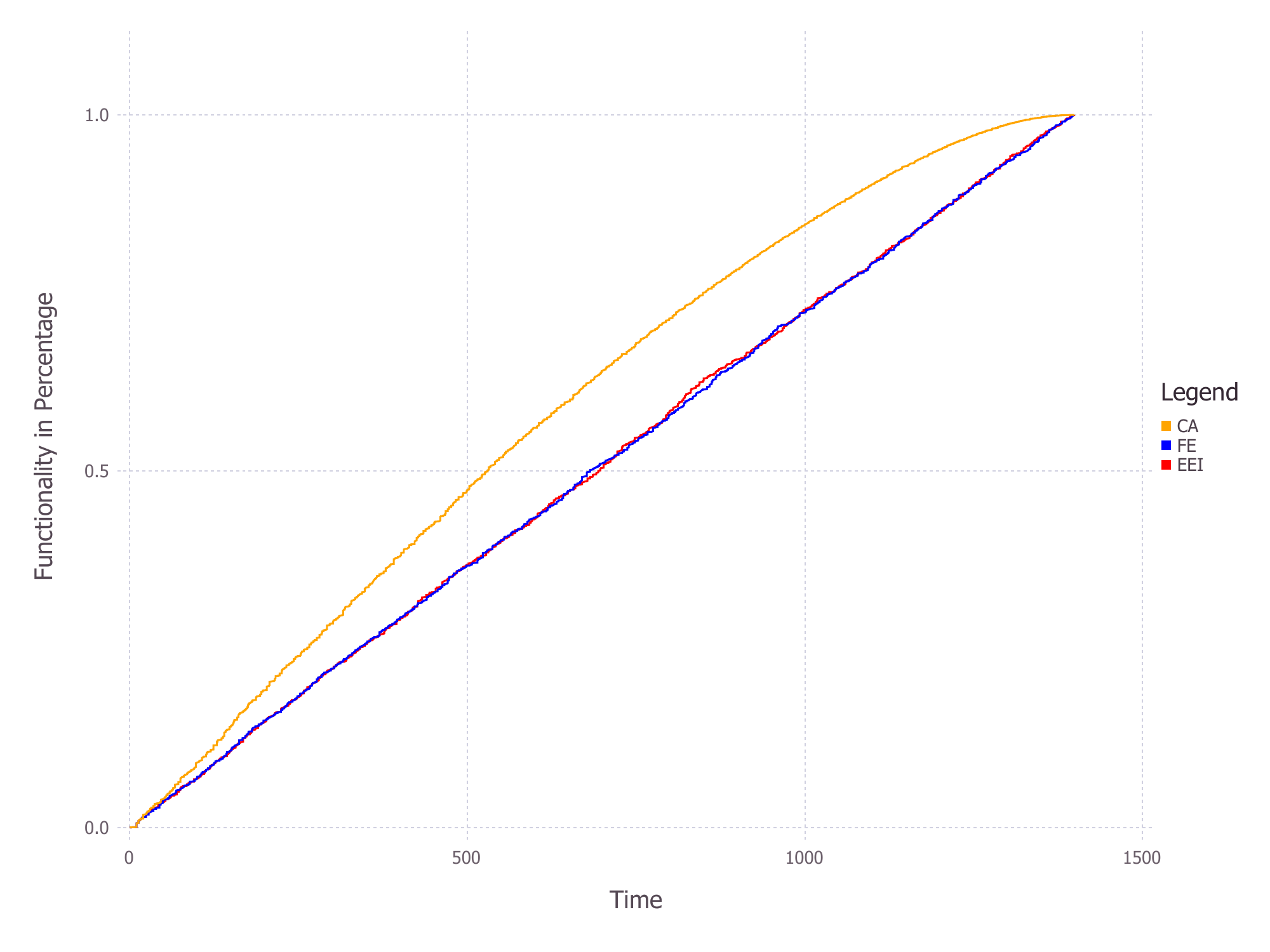}
\caption{Comparison of restoration trajectories: CA stands for our proposed $\rho$-factor based scheduling policy, FE for FirstEnergy Group's scheduling policy, and EEI for Edison Electric Institute's scheduling policy.}
\label{fig:traj_comparison}
\end{figure}
\subsection{Discussion}
From the three test cases above, we conclude that the ILP model would not be very useful for scheduling repairs and restoration in real time or near real time, except for very small problems. Even though it can be slow for large problems, the LP-based list scheduling algorithm can serve as an useful secondary tool for moderately sized problems. The conversion algorithm appears to have the best overall performance by far, in terms of solution quality and computing time.
\section{Conclusion}
In this paper, we investigated the problem of post-disaster repair and restoration in electricity distribution networks. We first proposed an ILP formulation which, although useful for benchmarking purposes, is feasible in practice only for small scale networks due to the immense computational time required to solve it to optimality or even near optimality. We then presented three heuristic algorithms. The first method, based on LP-relaxation of the ILP model, is proven to be a $2$-approximation algorithm. The second method converts the optimal single crew schedule, solvable in polynomial time, to an arbitrary $m$-crew schedule with a proven performance bound of $\left(2-\frac{1}{m}\right)$. The third method, based on $\rho$-factors which can be interpreted as component importance measures, is shown to be equivalent to the conversion algorithm. Simulations conducted on three IEEE standard networks indicate that the conversion algorithm provides very good results and is computationally efficient, making it suitable for real time implementation. The LP-based algorithm, while not as efficient, can still be used for small and medium scale problems.

Although we have focused on electricity distribution networks, the heuristic algorithms can also be applied to any infrastructure network with a radial structure (e.g., water distribution networks). Future work includes  development of efficient algorithms with proven approximation bounds which can be applied to arbitrary network topologies (e.g., meshed networks). While we have ignored transportation times between repair sites in this paper, this will be addressed in a subsequent paper. In fact, when repair jobs are relatively few and minor, but the repair sites are widely spread out geographically, optimal schedules are likely to be heavily influenced by the transportation times instead of the repair times. Finally, many  distribution networks contain switches that are normally open. These switches can be closed to restore power to some nodes from a different source. Doing so obviously reduces the aggregate harm. We intend to address this issue in the future.
%
\bibliography{det_res_ref}

\begin{thebibliography}{10}

\bibitem{eotp2013resiliency}
Economic benefits of increasing electric grid resilience to weather outages.
\newblock Technical report, 2013.

\bibitem{adibi1994power}
M.~M. Adibi and L.~H. Fink.
\newblock Power system restoration planning.
\newblock {\em IEEE Transactions on Power Systems}, 9(1):22--28, Feb 1994.

\bibitem{adibi2006overcoming}
M.~M. Adibi and L.~H. Fink.
\newblock Overcoming restoration challenges associated with major power system
  disturbances - restoration from cascading failures.
\newblock {\em IEEE Power and Energy Magazine}, 4(5):68--77, Sept 2006.

\bibitem{adolphson1973optimal}
D~Adolphson and T~Ch Hu.
\newblock Optimal linear ordering.
\newblock {\em SIAM Journal on Applied Mathematics}, 25(3):403--423, 1973.

\bibitem{gridwise2013resilience}
The~GridWise Alliance.
\newblock Improving electric grid reliability and resilience: Lessons learned
  from superstorm sandy and other extreme events.
\newblock Technical report, July 2013.

\bibitem{arritt2010the}
R.~F. Arritt and R.~C. Dugan.
\newblock The ieee 8500-node test feeder.
\newblock In {\em IEEE PES T D 2010}, pages 1--6, April 2010.

\bibitem{brucker2007scheduling}
Peter Brucker.
\newblock {\em Scheduling algorithms}, volume~3.
\newblock Springer, 2007.

\bibitem{chekuri2004approximation}
Chandra Chekuri and Sanjeev Khanna.
\newblock Approximation algorithms for minimizing average weighted completion
  time.
\newblock In {\em Handbook of Scheduling: Algorithms, Models, and Performance
  Analysis}. CRC Press, 2004.

\bibitem{coffrin2014transmission}
Carleton Coffrin and Pascal Van~Hentenryck.
\newblock Transmission system restoration: Co-optimization of repairs, load
  pickups, and generation dispatch.
\newblock In {\em Power Systems Computation Conference (PSCC), 2014}, pages
  1--8. IEEE, 2014.

\bibitem{graham1979optimization}
Ronald~L Graham, Eugene~L Lawler, Jan~Karel Lenstra, and AHG~Rinnooy Kan.
\newblock Optimization and approximation in deterministic sequencing and
  scheduling: a survey.
\newblock {\em Annals of discrete mathematics}, 5:287--326, 1979.

\bibitem{FirstEnergypractice}
FirstEnergy Group.
\newblock {\em {Storm Restoration Process}}.

\bibitem{horn1972single}
WA~Horn.
\newblock Single-machine job sequencing with treelike precedence ordering and
  linear delay penalties.
\newblock {\em SIAM Journal on Applied Mathematics}, 23(2):189--202, 1972.

\bibitem{hou2011computation}
Yunhe Hou, Chen-Ching Liu, Kai Sun, Pei Zhang, Shanshan Liu, and Dean Mizumura.
\newblock Computation of milestones for decision support during system
  restoration.
\newblock In {\em 2011 IEEE Power and Energy Society General Meeting}, pages
  1--10. IEEE, 2011.

\bibitem{EEIpractice}
Edison~Electric Institute.
\newblock {\em {Understanding the Electric Power Industry's Response and
  Restoration Process}}.

\bibitem{kersting2001radial}
W.~H. Kersting.
\newblock Radial distribution test feeders.
\newblock In {\em 2001 IEEE Power Engineering Society Winter Meeting.
  Conference Proceedings (Cat. No.01CH37194)}, volume~2, pages 908--912 vol.2,
  2001.

\bibitem{eidinger2012christchurch}
Alexis Kwasinski, John Eidinger, Alex Tang, and Christophe Tudo-Bornarel.
\newblock Performance of electric power systems in the 2010 - 2011
  christchurch, new zealand, earthquake sequence.
\newblock {\em Earthquake Spectra}, 30(1):205--230, 2014.

\bibitem{lawler1978sequencing}
Eugene~L Lawler.
\newblock Sequencing jobs to minimize total weighted completion time subject to
  precedence constraints.
\newblock {\em Annals of Discrete Mathematics}, 2:75--90, 1978.

\bibitem{ma1992operational}
T-K Ma, C-C Liu, M-S Tsai, R~Rogers, SL~Muchlinski, and J~Dodge.
\newblock Operational experience and maintenance of online expert system for
  customer restoration and fault testing.
\newblock {\em IEEE Transactions on Power Systems}, 7(2):835--842, 1992.

\bibitem{nerc2014sandy}
NERC.
\newblock Hurricane sandy event analysis report.
\newblock Technical report, January 2014.

\bibitem{nurre2012restoring}
Sarah~G Nurre, Burak Cavdaroglu, John~E Mitchell, Thomas~C Sharkey, and
  William~A Wallace.
\newblock Restoring infrastructure systems: An integrated network design and
  scheduling (inds) problem.
\newblock {\em European Journal of Operational Research}, 223(3):794--806,
  2012.

\bibitem{newyork2013resilient}
The~City of~New~York.
\newblock A stronger, more resilient new york.
\newblock Technical report, Jun 2013.

\bibitem{perez2008optimal}
Ra{\'u}l P{\'e}rez-Guerrero, Gerald~Thomas Heydt, Nevida~J Jack, Brian~K Keel,
  and Armindo~R Castelhano.
\newblock Optimal restoration of distribution systems using dynamic
  programming.
\newblock {\em IEEE Transactions on Power Delivery}, 23(3):1589--1596, 2008.

\bibitem{pinedo2012scheduling}
Michael~L Pinedo.
\newblock {\em Scheduling: theory, algorithms, and systems}.
\newblock Springer Science \& Business Media, 2012.

\bibitem{queyranne1993structure}
Maurice Queyranne.
\newblock Structure of a simple scheduling polyhedron.
\newblock {\em Mathematical Programming}, 58(1-3):263--285, 1993.

\bibitem{queyranne2006approximation}
Maurice Queyranne and Andreas~S Schulz.
\newblock Approximation bounds for a general class of precedence constrained
  parallel machine scheduling problems.
\newblock {\em SIAM Journal on Computing}, 35(5):1241--1253, 2006.

\bibitem{Reed:2009ekbaca}
D.~A. Reed, K.~C. Kapur, and R.~D. Christie.
\newblock {Methodology for Assessing the Resilience of Networked
  Infrastructure}.
\newblock {\em IEEE Systems Journal}, 3(2):174--180, June 2009.

\bibitem{Cimellaro:2010dpbaca}
Andrei~M. Reinhorn, Michel Bruneau, and Gian~Paolo Cimellaro.
\newblock {Framework for analytical quantification of disaster resilience}.
\newblock {\em Engineering Structures}, 32(11):3639--3649, November 2010.

\bibitem{schulz1996polytopes}
Andreas~S Schulz et~al.
\newblock {\em Polytopes and scheduling}.
\newblock PhD thesis, Citeseer, 1996.

\bibitem{doe2012derecho}
Infrastructure Security, Office of Electricity~Delivery Energy~Restoration, and
  U.S. Department of~Energy Energy~Reliability.
\newblock A review of power outages and restoration following the june 2012
  derecho.
\newblock Technical report, Aug 2012.

\bibitem{sidney1975decomposition}
Jeffrey~B Sidney.
\newblock Decomposition algorithms for single-machine sequencing with
  precedence relations and deferral costs.
\newblock {\em Operations Research}, 23(2):283--298, 1975.

\bibitem{smith1956various}
Wayne~E Smith.
\newblock Various optimizers for single-stage production.
\newblock {\em Naval Research Logistics Quarterly}, 3(1-2):59--66, 1956.

\bibitem{sun2011optimal}
Wei Sun, Chen-Ching Liu, and Li~Zhang.
\newblock Optimal generator start-up strategy for bulk power system
  restoration.
\newblock {\em IEEE Transactions on Power Systems}, 26(3):1357--1366, 2011.

\bibitem{toune2002comparative}
Sakae Toune, Hiroyuki Fudo, Takamu Genji, Yoshikazu Fukuyama, and Yosuke
  Nakanishi.
\newblock Comparative study of modern heuristic algorithms to service
  restoration in distribution systems.
\newblock {\em IEEE Transactions on Power Delivery}, 17(1):173--181, 2002.

\bibitem{Vugrin:2011kgba}
Eric~D. Vugrin, Drake~E. Warren, and Mark~A. Ehlen.
\newblock {A resilience assessment framework for infrastructure and economic
  systems: Quantitative and qualitative resilience analysis of petrochemical
  supply chains to a hurricane}.
\newblock {\em Process Safety Progress}, 30(3):280--290, September 2011.

\bibitem{wangresearch}
Y.~Wang, C.~Chen, J.~Wang, and R.~Baldick.
\newblock Research on resilience of power systems under natural disasters - a
  review.
\newblock {\em Power Systems, IEEE Transactions on}, PP(99):1--10, 2015.

\bibitem{xu2007optimizing}
Ningxiong Xu, Seth~D Guikema, Rachel~A Davidson, Linda~K Nozick, Zehra
  {\c{C}}a{\u{g}}nan, and Kabeh Vaziri.
\newblock Optimizing scheduling of post-earthquake electric power restoration
  tasks.
\newblock {\em Earthquake engineering \& structural dynamics}, 36(2):265--284,
  2007.

\bibitem{yamangil2014designing}
Emre Yamangil, Russell Bent, and Scott Backhaus.
\newblock Designing resilient electrical distribution grids.
\newblock {\em arXiv preprint arXiv:1409.4477}, 2014.

\end{thebibliography}
\end{document}